\documentclass[11pt,a4paper]{article}
\usepackage[cp1251]{inputenc}
\usepackage{amsmath,amsthm}
\usepackage{amsfonts,amstext,amssymb,verbatim,epsfig}
\usepackage{dsfont}
\usepackage{enumerate}
\usepackage{psfrag}
\usepackage{graphicx}
\usepackage{color}
\usepackage{float}
\usepackage{subfig}

\usepackage[top=3cm, bottom=3.5cm, left=2.5cm, right=2.5cm]{geometry}

\def\N{{\mathbb{N}}}
\def\Z{{\mathbb{Z}}}
\def\P{{\mathbb{P}}}

\def\slab{{\mathbb S}}
\def\an{{\mathrm{An}}}
\def\qm{c_*}

\newtheorem{theorem}{Theorem}[section]
\newtheorem{lemma}[theorem]{Lemma}

\newtheoremstyle{likedef}
  {}%
  {}%
  {}%
  {}
  {\bfseries}%
  {.}%
  {.5em}%
  {}%

\theoremstyle{likedef}

\newtheorem{remark}[theorem]{Remark}

\numberwithin{equation}{section}

\begin{document}

\title{Kesten's incipient infinite cluster and quasi-multiplicativity of crossing probabilities}

\author{
Deepan Basu
\thanks{Max-Planck Institute for Mathematics in the Sciences, 
Inselstrasse 22, 04103 Leipzig, Germany. 
email: deepan.basu@mis.mpg.de}
\and
Artem Sapozhnikov
\thanks{University of Leipzig, Department of Mathematics, 
Augustusplatz 10, 04109 Leipzig, Germany.
email: artem.sapozhnikov@math.uni-leipzig.de}
}

\maketitle

\begin{abstract}
In this paper we consider Bernoulli percolation on an infinite connected bounded degrees graph $G$. 
Assuming the uniqueness of the infinite open cluster and a quasi-multiplicativity of crossing probabilities, 
we prove the existence of Kesten's incipient infinite cluster. 
We show that our assumptions are satisfied if $G$ is a slab $\Z^2\times\{0,\ldots,k\}^{d-2}$ ($d\geq 2$, $k\geq 0$). 
We also argue that the quasi-multiplicativity assumption is fulfilled for $G=\Z^d$ if and only if $d<6$. 
\end{abstract}

\section{Introduction}

Let $G$ be an infinite connected bounded degrees graph with a vertex set $V$. Let $\rho$ be the graph metric on $V$, 
and define for $v\in V$ and positive integers $m\leq n$, 
\[
B(v,n) = \{x\in V~:~\rho(v,x)\leq n\},\quad 
S(v,n) = \{x\in V~:~\rho(v,x) = n\},
\]
\[
A(v,m,n) = B(v,n)\setminus B(v,m-1).
\]

\smallskip

Consider Bernoulli bond percolation on $G$ with parameter $p\in[0,1]$ and denote the corresponding probability measure by $\P_p$. 
The open cluster of $v\in V$ is denoted by $C(v)$. 
Let $p_c$ be the critical threshold for percolation, i.e., for $v\in V$, 
\[
p_c = \inf\left\{p:\P_p[|C(v)|=\infty]>0\right\}.
\]
For $x,y\in V$ and $X,Y,Z\subset V$, we write $x\leftrightarrow y$ in $Z$ if there is a nearest neighbor path of open edges such that all its vertices are in $Z$, 
$X\leftrightarrow Y$ in $Z$ if there exist $x\in X$ and $y\in Y$ such that $x\leftrightarrow y$ in $Z$, and $x\leftrightarrow Y$ in $Z$, if 
there exist $y\in Y$ such that $x\leftrightarrow y$ in $Z$. 
If $Z=V$, we omit ``in $Z$'' from the notation. 
We use $\nleftrightarrow$ instead of $\leftrightarrow$ to denote complements of the respective events. 

\medskip

In this note we are interested in the existence and equality of the limits
\begin{equation}\label{eq:limits}
\lim_{n\to\infty}\P_{p_c}\left[E~|~w\longleftrightarrow S(w,n)\right]\quad\text{and}\quad
\lim_{p\searrow p_c}\P_p\left[E~|~|C(w)|=\infty\right],
\end{equation}
where $E$ is a cylinder event. The question is highly non-trivial if $\P_{p_c}\left[|C(w)|=\infty\right]=0$. 
The seminal result of Kesten \cite[Theorem~(3)]{KestenIIC} states that if $G$ is from a class of two dimensional graphs, such as $\Z^2$, 
then the above two limits exist and have the same value $\nu_{G,w}(E)$.  
By Kolmogorov's extension theorem, $\nu_{G,w}$ extends uniquely to a probability measure on configurations of edges, which is often called
{\it Kesten's incipient infinite cluster} measure. It is immediate that $\nu_{G,w}[|C(w)|=\infty] = 1$. 
Kesten's argument is based on the existence of an infinite collection of open circuits around $w$ in disjoint annuli 
and the properties that (a) each path from $w$ to infinity intersects every such circuit and 
(b) by conditioning on the innermost open circuit in an annulus, the occupancy configuration in the region not surrounded by the circuit is 
still an independent Bernoulli percolation. These properties are no longer valid when one considers higher dimensional lattices. 
In fact, the existence of Kesten's IIC on $\Z^d$ for $d\geq 3$ is still an open problem. 
A partial progress has been recently made in sufficiently high dimensions by Heydenreich, van der Hofstad and Hulshof \cite[Theorem~1.2]{HvdHH14}, 
who showed using lace expansions the existence of the first limit in \eqref{eq:limits} under the assumption that $n^{-2}\,\P_{p_c}[0\longleftrightarrow S(0,n)]$ converges. 
Concerning low dimensional lattices, almost nothing is known there about critical and near critical percolation, 
and the existence of Kesten's IIC seems particularly hard to show. 
Several other constructions of incipient infinite clusters are obtained by J\'arai \cite{Jarai} for planar lattices and 
van der Hofstad and J\'arai \cite{vdHJ04} for high dimensional lattices. 

\medskip

The main result of this note is the existence and the equality of the two limits in \eqref{eq:limits} for graphs satisfying two assumptions: 
(A1) uniqueness of the infinite open cluster and (A2) quasi-multiplicativity of crossing probabilities. 
While (A1) is satisfied by many amenable graphs, most notably $\Z^d$, 
(A2) can be expected only in low dimensional graphs. 
For instance, we argue below that (A2) holds for $\Z^d$ if and only if $d<6$. 
In our second result, we prove that (A2) is satisfied by slabs $\Z^2\times\{0,\ldots, k\}^{d-2}$ ($d\geq 2$, $k\geq 0$), 
thus showing for these graphs the existence and equality of the limits in \eqref{eq:limits}. 
We now state the assumptions and the main result, and then comment more on the assumptions. 

\medskip

\begin{itemize}
\item[{\bf (A1)}] (Uniqueness of the infinite open cluster) 
For any $p\in[0,1]$ there exists almost surely at most one infinite open cluster.
\item[{\bf (A2)}] (Quasi-multiplicativity of crossing probabilities)
Let $v\in V$ and $\delta>0$. There exists $\qm>0$ such that for any $p\in[p_c,p_c+\delta]$, integer $m>0$, a finite connected set $Z\subset V$ such that 
$Z\supseteq A(v,m, 4m)$, and sets $X\subset Z\cap B(v,m)$ and $Y\subset Z\setminus B(v,4m)$, 
\begin{equation}\label{eq:A2}
\P_p[X\leftrightarrow Y\text{ in }Z] \geq \qm\cdot \P_p[X\leftrightarrow S(v,2m)\text{ in }Z]\cdot \P_p[Y\leftrightarrow S(v,2m)\text{ in }Z].
\end{equation}
\end{itemize}
\begin{theorem}\label{thm:IIC}
Assume that the graph $G$ satisfies the assumptions (A1) and (A2) for some choice of $v\in V$ and $\delta>0$. 
Then, for any cylinder event $E$, the two limits in \eqref{eq:limits} exist and have the same value. 

If the assumptions (A1) and (A2) are satisfied at $p=p_c$, then the first limit in \eqref{eq:limits} exists. 
\end{theorem}

Before we discuss the strategy of the proof, let us comment on the assumptions. 
\paragraph{Comments on (A1):}
\begin{enumerate}
\item
(A1) is satisfied by many sufficiently regular (e.g., vertex transitive) amenable graphs, 
most notably lattices $\Z^d$ and slabs $\Z^2\times\{0,\ldots,k\}^{d-2}$ ($d\geq 2$, $k\geq 0$), see, e.g., \cite{BS96}. 
\item
(A1) is equivalent to the assumption that for some $\delta>0$ there exists at most one infinite open cluster for any fixed $p\in[p_c,p_c+\delta]$. 
Indeed, if for a given $p$ the infinite open cluster is unique almost surely, then 
the same holds for any $p'>p$, see, e.g., \cite{HP99,Sch99}.
\item
For $v\in V$ and $m\leq n$, let $E_1(v,m,n) = \{S(v,m)\leftrightarrow S(v,n)\}$ 
and $E_2(v,m,n)$ the event that in the annulus $A(v,m,n)$ there are at least two disjoint open crossing clusters.

Assumption (A1) is equivalent to the following one, which will be used in the proof of Theorem~\ref{thm:IIC}: 
For any $v\in V$, $\varepsilon>0$ and $m\in \N$, there exists $n>4m$ such that 
\begin{equation}\label{eq:E2mn}
\sup_{p\in[0,1]}\P_p\left[E_2(v,m,n)\right]<\varepsilon
\end{equation}
or, equivalently,
\begin{equation}\label{eq:E2mnE1mn}
\sup_{p\in[0,1]}\P_p\left[E_2(v,m,n)~|~E_1(v,m,n)\right]<\varepsilon.
\end{equation}
The equivalence of the claims \eqref{eq:E2mn} and \eqref{eq:E2mnE1mn} follows from the inequalities 
\[
\P_p\left[E_2(v,m,n)\right]\leq \P_p\left[E_2(v,m,n)~|~E_1(v,m,n)\right]\leq \P_p\left[E_2(v,m,n)\right]^{\frac 12},
\]
where the second one is a consequence of the BK inequality. 

It is elementary to see that \eqref{eq:E2mn} implies (A1). On the other hand, if \eqref{eq:E2mn} does not hold, 
then there exist $v_0\in V$, $\varepsilon_0>0$ and $m_0\in\N$ such that for all $n>4m_0$, 
$\sup_{p\in[0,1]}\P_p\left[E_2(v_0,m_0,n)\right]\geq \varepsilon_0$.
The function $\P_p\left[E_2(v_0,m_0,n)\right]$ is continuous in $p\in[0,1]$ and monotone decreasing in $n$. 
Thus, there exists $p_0\in[0,1]$ such that $\P_{p_0}\left[E_2(v_0,m_0,n)\right]\geq \varepsilon_0$ for all $n>4m_0$. 
By passing to the limit as $n\to\infty$, we conclude that for $p=p_0$, with positive probability there exist at least two infinite open clusters and (A1) does not hold. 
\end{enumerate}

\paragraph{Comments on (A2):}
\begin{enumerate}
\setcounter{enumi}{3}
\item
It follows from the Russo-Seymour-Welsh Theorem \cite{R78,SW78} that (A2) holds for two dimensional graphs, such as $\Z^2$, considered by Kesten in \cite{KestenIIC}. 
Russo-Seymour-Welsh ideas have been recently extended to slabs in \cite{NTW15,BS15}, after the absence of percolation at criticality in slabs was proved 
by Duminil-Copin, Sidoravicius and Tassion \cite{DCST14}. 
In Lemma~\ref{l:main} of the present paper we prove that (A2) is fulfilled by slabs $\Z^2\times\{0,\ldots, k\}^{d-2}$ ($d\geq 2$, $k\geq 0$), 
thus verifying the existence and equality of the limits \eqref{eq:limits} for slabs. 
\item
We believe that assumption (A2) holds for lattices $\Z^d$ if $d<6$, but does not hold if $d>6$. 
Dimension $d_c=6$ is called the {\it upper critical dimension} above which the percolation phase transition 
should be described by mean-field theory, see, e.g., \cite{CC87}. 
This was rigorously confirmed in sufficiently high dimensions by Hara and Slade \cite{HS90, Hara}. 

\smallskip

It is easy to see that the mean-field behavior excludes (A2). Indeed, 
it is believed that above $d_c$, the two point function decays as 
\[
\P_{p_c}[x\leftrightarrow y] \asymp (1+\rho(x,y))^{2-d}.
\]
(Here $f(z)\asymp g(z)$ if for some $c$, $cf(z)\leq g(z)\leq c^{-1}f(z)$ for all $z$.)
Hara \cite{Hara} proved it rigorously in sufficiently high dimensions. 
Given this asymptotics, Aizenman showed in \cite[Theorem~4(2)]{A97} that for all $m(n)\leq n$ such that $\frac{m(n)}{n^{2/(d-4)}} \to \infty$, 
\[
\P_{p_c}\left[S(0,m(n))\leftrightarrow S(0,n)\right] \to 1,\quad\text{as }n\to \infty,
\]
and Kozma and Nachmias \cite{KN11} that $\P_{p_c}\left[0\leftrightarrow S(0,n)\right]\asymp n^{-2}$. Thus, the inequality 
\[
\P_{p_c}[0\leftrightarrow S(0,n)]\geq c\,\P_{p_c}[0\leftrightarrow S(0,m(n))]\,\P_{p_c}[S(0,m(n))\leftrightarrow S(0,n)]
\]
cannot hold for large $n$. 

\smallskip

The situation below $d_c$ is much more subtle. With the exception of $d=2$, 
where planarity helps enormously, the (near-)critical behavior below $d_c$ 
is widely unknown. Let us nevertheless give a few words about why we think (A2) should hold below $d_c$. 
It is believed that the number of clusters crossing any annulus $A(0,m,2m)$ is bounded uniformly in $m$ 
if $d<d_c$ and grows at $p=p_c$ like $m^{d-6}$ above $d_c$, with log-correction for $d=d_c$, 
and this dichotomy is intimately linked to the transition at $d_c$ from the hyperscaling to the mean-field; see \cite{Con85,BCKS}. 
Thus, it would be not unreasonable to expect that below $d_c$, 
\[
\P_p[\exists!\,\text{crossing cluster of $A(0,m,2m)$}~|~X\leftrightarrow S(0,2m)\text{ in }Z,\,Y\leftrightarrow S(0,m)\text{ in }Z]\geq c >0,
\]
which is enough to establish (A2). We are not able to prove it yet or give a simpler sufficient condition for it. 
It would already be very nice if, for instance, 
(A2) was derived from the assumption that $\P_p[\exists!\,\text{crossing cluster of $A(0,m,2m)$}]\geq c$ 
or from the assumptions of \cite{BCKS}. 
\end{enumerate}

\smallskip

We finish the introduction with a brief description of the proof of Theorem~\ref{thm:IIC}. 
Our proof follows the general scheme proposed by Kesten in \cite{KestenIIC} 
by attempting to decouple the configuration near $w$ from infinity on multiple scales. 
The implementations are however rather different. 
Using \eqref{eq:E2mnE1mn} we identify a sufficiently fast growing sequence $N_i$ such that given $w\leftrightarrow S(w,n)$, 
the probability that the annulus $A(v,N_i,N_{i+1})\subset B(w,n)$ contains a unique crossing cluster is asymptotically close to $1$; see \eqref{eq:Fi:proba}. 
Next, let an annulus $A(v,N_i,N_{i+1})$ contain a unique crossing cluster. 
We explore all the open clusters in this annulus that intersect the interior boundary $S(v,N_i)$, call their union $\mathcal C_i$, 
and let $\mathcal D_i$ be the subset of $S(v,N_{i+1}+1)$ of vertices connected by an open edge to $\mathcal C_i$; see \eqref{eq:CD}. 
Then, the configuration outside $\mathcal C_i$ is distributed as the original independent percolation 
and every vertex from $\mathcal D_i$ is connected by an edge to the same (crossing) cluster from $\mathcal C_i$. 
Thus, $w\leftrightarrow S(w,n)$ if and only if (a) $w$ is connected to $\mathcal D_i$ (this event only depends on the edges intersecting $S(v,N_i)\cup \mathcal C_i$) 
and (b) $\mathcal D_i$ is connected to $S(w,n)$ outside $\mathcal C_i$ (this only depends on the edges outside $\mathcal C_i$). 
This allows to factorize $\P_p[E,\,w\leftrightarrow S(w,n)]$; see \eqref{eq:difference:1}. 
The rest of the proof is essentially the same as that of Kesten \cite{KestenIIC}. 
We repeat the described factorization on several scales, obtaining in \eqref{eq:ratio:main} an approximation of $\P_p[E|w\leftrightarrow S(w,n)]$ 
in terms of products of positive matrices. Finally, we use (A2) to prove that the matrix operators are uniformly contracting, which 
is enough to conclude the proof; see \eqref{eq:Mratio} and the text below.

\section{Proof of Theorem~\ref{thm:IIC}}

We will prove the first claim of the theorem. The proof of the second one follows from the proof below by replacing everywhere $p$ by $p_c$. 
The general outline of the proof is the same as the original one of Kesten \cite[Theorem (3)]{KestenIIC}, but the choice of scales and the decoupling are done differently. 

\smallskip

First of all, it suffices to prove that for any $w\in V$ and a cylinder event $E$, 
\begin{equation}\label{eq:uniformconvergence}
\text{$\P_p[E|w\leftrightarrow S(w,n)]$ converges to some $\nu_p(E)$ {\it uniformly} on $[p_c,p_c+\delta]$ 
for some $\delta>0$.}
\end{equation}
Indeed, \eqref{eq:uniformconvergence} implies the existence of the first limit in \eqref{eq:limits} and that $\nu_p(E)$ is continuous. 
Since for any $p>p_c$, $\nu_p(E) = \P_p[E~|~|C(v)|=\infty]$, 
the existence of the second limit in \eqref{eq:limits} and its equality to the first one follows from the continuity of $\nu_p(E)$. 

\smallskip

Actually, by the inclusion-exclusion formula, it suffices to prove \eqref{eq:uniformconvergence} for all 
events $E$ of the form $\{\text{edges }e_1,\ldots,e_k\text{ are open}\}$. 
Although our proof could be implemented for any cylinder event $E$, calculations are neater for increasing events. 

\smallskip

Fix $w\in V$ and an increasing event $E$. 
Also fix $v\in V$ and $\delta>0$ for which the assumption (A2) is satisfied. 
Consider a sequence of scales $N_i$ such that $N_{i+1}>4N_i$ for all $i$, $B(v,N_0)$ contains $w$ and the states of its edges determine $E$. 
We will write $B_i = B(v,N_i)$, $S_i = S(v,N_i)$ and $A_i = A(v,N_i,N_{i+1})$. 
Let $F_i$ be the event that there exists a unique open crossing cluster in $A_i$. 
Define
\[
\varepsilon_i = \sup_{p\in[p_c,p_c+\delta]}\P_p\left[F_i^c~|~S_i\leftrightarrow S_{i+1}\right].
\]
By \eqref{eq:E2mnE1mn}, we can choose the scales $N_i$ so that $\varepsilon_i\to 0$ as $i\to\infty$. 

We first note that for $n>N_{i+1}+N_0$,
\begin{equation}\label{eq:Fi:proba}
\P_p[w\leftrightarrow S(w,n), F_i^c] \leq \qm^{-2}\varepsilon_i\cdot \P_p[w\leftrightarrow S(w,n)],
\end{equation}
where $\qm$ is the constant in the assumption (A2). Indeed, by independence, 
\begin{eqnarray*}
\P_p[w\leftrightarrow S(w,n), F_i^c] &\leq &\P_p[w\leftrightarrow S_i]\cdot \P_p[S_i\leftrightarrow S_{i+1}, F_i^c]\cdot \P_p[S_{i+1}\leftrightarrow S(w,n)]\\
&\leq &\varepsilon_i\cdot \P_p[w\leftrightarrow S_i]\cdot \P_p[S_i\leftrightarrow S_{i+1}]\cdot \P_p[S_{i+1}\leftrightarrow S(w,n)]\\
&\leq &\qm^{-2}\varepsilon_i\cdot \P_p\left[w\leftrightarrow S(w,n)\right],
\end{eqnarray*}
where the last inequality follows from the assumption (A2). 

\medskip

We begin to describe the main decomposition step. 
Consider the random sets
\begin{equation}\label{eq:CD}
\begin{aligned}
\mathcal C_i &= \left\{x\in B(v,N_{i+1})~:~x\leftrightarrow B(v,N_i)\text{ in }B(v,N_{i+1})\right\},\\
\mathcal D_i &= \left\{x\in S(v,N_{i+1}+1)~:~\text{$\exists\, y\in \mathcal C_i$, a neighbor of $x$, such that edge $\langle x,y\rangle$ is open}\right\}.
\end{aligned}
\end{equation}
Note that $\mathcal C_i$ contains $B(v,N_i)$, the event $\{\mathcal C_i = U\}$ depends only on the states of edges in $B(v,N_{i+1})$ with at least one end-vertex in $U$, 
and either $\{\mathcal C_i = U\}\subset F_i$ or $\{\mathcal C_i = U\}\cap F_i = \emptyset$. 
Also note that the event $\{\mathcal C_i = U,\, \mathcal D_i = R\}$ depends only on the states of edges in $B(v,N_{i+1}+1)$ with at least one end-vertex in $U$. 

For any $U\subset B(v,N_{i+1})$ and $R\subset S(v,N_{i+1}+1)$, consider the event 
\[
F_i(U,R) = \{\mathcal C_i = U,\, \mathcal D_i = R\},
\]
and let $\Pi_i$ be the collection of all such pairs $(U,R)$ that $\{\mathcal C_i = U\}\subset F_i$ and $F_i(U,R)\neq\emptyset$. 
Then $F_i = \cup_{(U,R)\in \Pi_i} F_i(U,R)$, and for all $n>N_{i+1}+N_0$, 
\begin{multline*}
\P_p\left[E,w\leftrightarrow S(w,n), F_i\right]
= \sum_{(U,R)\in \Pi_i}\P_p\left[E,w\leftrightarrow S(w,n), F_i(U,R)\right]\\
= \sum_{(U,R)\in \Pi_i}\P_p\left[E,w\leftrightarrow S_{i+1}, F_i(U,R)\right]\cdot\P_p\left[R\leftrightarrow S(w,n)\text{ in }B(w,n)\setminus U\right].
\end{multline*}
Together with \eqref{eq:Fi:proba}, this gives the inequality
\begin{multline}\label{eq:difference:1}
\Big|\P_p\left[E,w\leftrightarrow S(w,n)\right] - 
\sum_{(U,R)\in \Pi_i}\P_p\left[E,w\leftrightarrow S_{i+1}, F_i(U,R)\right]\cdot\P_p\left[R\leftrightarrow S(w,n)\text{ in }B(w,n)\setminus U\right]\Big|\\
\leq \qm^{-2}\varepsilon_i\cdot \P_p[w\leftrightarrow S(w,n)]
\leq \frac{\qm^{-2}\varepsilon_i}{\P_{p_c}[E]}\cdot \P_p[E,w\leftrightarrow S(w,n)],
\end{multline}
where the last step follows from the FKG inequality, since $E$ is increasing. 
Define the constant $C_* = (\qm^2\,\P_{p_c}[E])^{-1}$ and for $(U,R)\in\Pi_i$, let
\begin{align*}
u_p'(U,R) &= \P_p\left[E,w\leftrightarrow S_{i+1}, F_i(U,R)\right],\\
u_p''(U,R) &= \P_p\left[w\leftrightarrow S_{i+1}, F_i(U,R)\right],\\
\gamma_p(U,R,n) &= \P_p\left[R\leftrightarrow S(w,n)\text{ in }B(w,n)\setminus U\right].
\end{align*}
In this notation, \eqref{eq:difference:1} becomes
\[
\left(1 - C_*\varepsilon_i\right)\,\P_p\left[E,w\leftrightarrow S(w,n)\right]
\leq
\sum_{(U,R)\in \Pi_i}u_p'(U,R)\,\gamma_p(U,R,n)
\leq
\left(1 + C_*\varepsilon_i\right)\,\P_p\left[E,w\leftrightarrow S(w,n)\right]
\]
and by replacing $E$ above with the sure event, we also get
\[
\left(1 - C_*\varepsilon_i\right)\,\P_p\left[w\leftrightarrow S(w,n)\right]
\leq
\sum_{(U,R)\in \Pi_i}u_p''(U,R)\,\gamma_p(U,R,n)
\leq
\left(1 + C_*\varepsilon_i\right)\,\P_p\left[w\leftrightarrow S(w,n)\right].
\]

\medskip

Now we iterate. Let $(U,R)\in \Pi_i$. We can apply a similar reasoning as in \eqref{eq:Fi:proba} and \eqref{eq:difference:1} 
to $\gamma_p(U,R,n)$ and obtain that for any $j>i+2$ and $n> N_{j+1}+N_0$, 
\begin{multline}\label{eq:difference:2}
\Big|\gamma_p(U,R,n) - 
\sum_{(U',R')\in \Pi_j}\P_p\left[R\leftrightarrow S_{j+1}\text{ in }B_{j+1}\setminus U, F_{j-1}, F_j(U',R')\right]
\cdot\gamma_p(U',R',n)\Big|\\
\leq \qm^{-2}(\varepsilon_{j-1} + \varepsilon_j)\cdot \gamma_p(U,R,n).
\end{multline}
For $j> i+2$, $(U,R)\in \Pi_i$ and $(U',R')\in \Pi_j$, define 
\[
M_p(U,R;\,U',R') = \P_p\left[R\leftrightarrow S_{j+1}\text{ in }B_{j+1}\setminus U, F_{j-1}, F_j(U',R')\right].
\]
Then \eqref{eq:difference:2} becomes
\begin{multline*}
(1- \qm^{-2}\,(\varepsilon_{j-1} + \varepsilon_j))\,\gamma_p(U,R,n)
\leq
\sum_{(U',R')\in \Pi_j}M_p(U,R;\,U',R')\,\gamma_p(U',R',n)\\
\leq 
(1+ \qm^{-2}\,(\varepsilon_{j-1} + \varepsilon_j))\,\gamma_p(U,R,n).
\end{multline*}
Iterating further gives that for any $\varepsilon>0$ and $s\in\N$, there exist indices $i_1,\ldots, i_s$ such that $i_{k+1}>i_k + 2$ and 
for all $n> N_{i_s+1} + N_0$, 
\begin{multline}\label{eq:ratio:main}
e^{-\varepsilon}\,\P_p\left[E\,|\,w\leftrightarrow S(w,n)\right]\leq\\
\frac{\sum\, u_p'(U_1,R_1)\,M_p(U_1,R_1;\,U_2,R_2)\ldots M_p(U_{s-1},R_{s-1};, U_s,R_s)\,\gamma_p(U_s,R_s,n)}
{\sum\,u_p''(U_1,R_1)\,M_p(U_1,R_1;\,U_2,R_2)\ldots M_p(U_{s-1},R_{s-1};, U_s,R_s)\,\gamma_p(U_s,R_s,n)}\\
\leq e^{\varepsilon}\,\P_p\left[E\,|\,w\leftrightarrow S(w,n)\right],
\end{multline}
where the two sums are over $(U_1,R_1)\in \Pi_{i_1},\ldots, (U_s,R_s)\in \Pi_{i_s}$.

We will prove that (A2) implies that there exists $\kappa$ such that for all $i$, $j>i+2$, all pairs $(U_1,R_1), (U_2,R_2)\in \Pi_i$, $(U_1',R_1'), (U_2',R_2')\in \Pi_j$, 
and all $p\in[p_c,p_c+\delta]$, 
\begin{equation}\label{eq:Mratio}
\frac{M_p(U_1,R_1;\,U_1',R_1')\,M_p(U_2,R_2;\,U_2',R_2')}{M_p(U_1,R_1;\,U_2',R_2')\,M_p(U_2,R_2;\,U_1',R_1')}\leq \kappa^2.
\end{equation}
(This is an analogue of \cite[Lemma (23)]{KestenIIC}.) If so, then we can use Hopf's contraction property of multiplication by positive matrices 
as in \cite[pages 377-378]{KestenIIC}\footnotemark[1]\footnotetext[1]{There is a mathematical typo in the first inequality on \cite[page 378]{KestenIIC} -- $\mathrm{osc}(u',u'')$ is missing. 
However, one can show using RSW techniques that the missing term there is bounded from above by a constant independent of $j_1$, and the remaining argument goes through. 
In our case, the situation is simpler, since for our choice of $u'$ and $u''$, $\mathrm{osc}(u',u'')\leq 1$.}
to conclude from \eqref{eq:ratio:main} that there exists $\xi\leq 1$, which depends on $E$, $p$, and the scales $i_1,\ldots, i_s$, such that 
for all $n> N_{i_s+1} + N_0$, 
\begin{equation}\label{eq:inequality:xi}
e^{-\varepsilon}\left(\xi - \left(\frac{\kappa - 1}{\kappa + 1}\right)^{s-1}\right)
\leq 
\P_p\left[E\,|\,w\leftrightarrow S(w,n)\right]
\leq 
e^{\varepsilon}\left(\xi + \left(\frac{\kappa - 1}{\kappa + 1}\right)^{s-1}\right).
\end{equation}
It follows from \eqref{eq:inequality:xi} and the fact that $\xi\leq 1$ that for any $m,n> N_{i_s+1} + N_0$ and $p\in[p_c,p_c+\delta]$, 
\[
\Big|\P_p\left[E\,|\,w\leftrightarrow S(w,m)\right] - \P_p\left[E\,|\,w\leftrightarrow S(w,n)\right]\Big|\\
\leq
\left(e^{\varepsilon} - e^{-\varepsilon}\right) + \left(e^{\varepsilon} + e^{-\varepsilon}\right)\,\left(\frac{\kappa - 1}{\kappa + 1}\right)^{s-1},
\]
which implies \eqref{eq:uniformconvergence}. 

\medskip

It remains to prove \eqref{eq:Mratio}. Let $j>i+2$. Consider the random sets
\begin{align*}
\mathcal X_j &= \left\{x\in A_{j-1}~:~x\leftrightarrow S_j\text{ in }A_{j-1}\right\},\\
\mathcal Y_j &= \left\{y\in S(v,N_{j-1}-1)~:~\text{$\exists\, x\in \mathcal X_j$, a neighbor of $y$, such that the edge $\langle x,y\rangle$ is open}\right\}.
\end{align*}
Note that $\mathcal X_j$ contains $S_j$, the event $\{\mathcal X_j = X\}$ depends only on the states of edges in $A_{j-1}$ with at least one end-vertex in $X$, 
and either $\{\mathcal X_j = X\}\subset F_{j-1}$ or $\{\mathcal X_j = X\}\cap F_{j-1} = \emptyset$. 
Also note that the event $\{\mathcal X_j = X,\, \mathcal Y_j = Y\}$ depends only on the states of edges in $B_j$ with at least one end-vertex in $X$. 
For any $X\subset A_{j-1}$ and $Y\subset S(v,N_{j-1}-1)$, consider the event 
\[
G_j(X,Y) = \{\mathcal X_j = X,\, \mathcal Y_j = Y\},
\]
and let $\Gamma_j$ be the collection of all such pairs $(X,Y)$ that $\{\mathcal X_j = X\}\subset F_{j-1}$ and $G_j(X,Y)\neq\emptyset$. 
Then $F_{j-1} = \cup_{(X,Y)\in \Gamma_j} G_j(X,Y)$ and for any $(U,R)\in\Pi_i$, $(U',R')\in\Pi_j$,
\[
M_p(U,R;\,U',R') = 
\sum_{(X,Y)\in \Gamma_j} \P_p\left[R\leftrightarrow Y \text{ in }B_j\setminus (X\cup U)\right]\cdot 
\P_p\left[G_j(X,Y), F_j(U',R'), Y\leftrightarrow R'\right].
\]
By the assumption (A2), 
\begin{multline*}
\P_p\left[R\leftrightarrow Y \text{ in }B_j\setminus (X\cup U)\right]\\
\geq \qm\cdot\P_p\left[R\leftrightarrow S(v,2N_{i+1}) \text{ in }B(v,2N_{i+1})\setminus U\right]
\cdot 
\P_p\left[S(v,2N_{i+1})\leftrightarrow Y \text{ in }B_j\setminus X\right]\\
\geq 
\qm\cdot\P_p\left[R\leftrightarrow Y \text{ in }B_j\setminus (X\cup U)\right].
\end{multline*}
This easily implies \eqref{eq:Mratio} with $\kappa = \qm^{-1}$.
The proof of Theorem~\ref{thm:IIC} is complete. 
\qed

\medskip

\begin{remark}\label{rem:JaraiIIC}
Instead of conditioning on the events $\{w\leftrightarrow S(w,n)\}$, one could condition on $\{w\leftrightarrow Y_n\text{ in }Z_n\}$, 
where $Z_n\supset B(w,n)$ and $Y_n\subseteq Z_n\setminus B(w,n)$, and obtain the same limits as in \eqref{eq:limits}. 
This is immediate after observing that $\P_p[E|w\leftrightarrow Y_n\text{ in }Z_n]$ satisfies inequalities \eqref{eq:inequality:xi} 
with the same $\xi$.
\end{remark}

\bigskip
\bigskip
\bigskip

\section{Quasi-multiplicativity for slabs}

In this seciton we prove that the assumption (A2) is fulfilled by slabs $\Z^2\times\{0,\ldots,k\}^{d-2}$ for any $d\geq 2$ and $k\geq 0$ and 
for any $\delta>0$ such that $p_c+\delta<1$, thus proving
\begin{theorem}\label{thm:slabs}
The two limits in \eqref{eq:limits} exist and coincide for $\Z^2\times\{0,\ldots,k\}^{d-2}$ ($d\geq 2$, $k\geq 0$).
\end{theorem}

\smallskip

Fix $d\geq 2$ and $k\geq 0$ and define $\slab = \Z^2\times\{0,\ldots,k\}^{d-2}$. 
For positive integers $m\leq n$, let $Q(n) = [-n,n]^2\times\{0,\ldots,k\}^{d-2}$ be the box of side length $2n$ in $\slab$ centered at $0$, 
$\partial Q(n) = Q(n)\setminus Q(n-1)$ the inner boundary of $Q(n)$, and $\an(m,n) = Q(n)\setminus Q(m-1)$ the annulus of side lengths $2m$ and $2n$. 
We will prove the following lemma. 
\begin{lemma}\label{l:main}
Let $d\geq 2$ and $k\geq 0$. Let $\delta>0$ such that $p_c + \delta<1$. 
There exists $c>0$ such that for any $p\in[p_c,p_c+\delta]$, integer $m>0$, 
any finite connected $Z\subset \slab$ such that $Z\supseteq \an(m,3m)$, and any $X\subset Z\cap Q(m)$ and $Y\subset Z\setminus Q(3m)$, 
\begin{equation}\label{eq:main}
\P_p[X\leftrightarrow Y\text{ in }Z] \geq c\cdot \P_p[X\leftrightarrow \partial Q(2m)\text{ in }Z]\cdot \P_p[Y\leftrightarrow \partial Q(2m)\text{ in }Z].
\end{equation}
\end{lemma}
To see that Lemma~\ref{l:main} implies (A2), note that it suffices to prove \eqref{eq:A2} for $m\geq m_0$ and sufficiently large $m_0$. 
One can choose $m_0 = m_0(d,k)$ large enough so that $A(0,m,4m)\supset \an(m, 3m)$. Thus, Lemma~\ref{l:main} implies (A2).

\begin{proof}[Proof of Lemma~\ref{l:main}]
Instead of \eqref{eq:main}, it suffices to prove that there exists $c>0$ such that for any $m>0$, 
any finite connected $Z\subset \slab$ such that $Z\supseteq \an(2m, 3m)$, and any $X\subset Z\cap Q(2m)$ and $Y\subset Z\setminus Q(3m)$, 
\begin{equation}\label{eq:main:proof}
\P_p[X\leftrightarrow Y\text{ in }Z] \geq c\cdot \P_p[X\leftrightarrow \partial Q(3m)\text{ in }Z]\cdot \P_p[Y\leftrightarrow \partial Q(2m)\text{ in }Z].
\end{equation}
Indeed, for $Z$ as in the statement of the lemma, by \eqref{eq:main:proof}, 
\[
\P_p[X\leftrightarrow \partial Q(3m)\text{ in }Z]\geq c\cdot \P_p[X\leftrightarrow \partial Q(2m)\text{ in }Z]\cdot \P_p[\partial Q(\frac43 m)\leftrightarrow \partial Q(3m)\text{ in }Z], 
\]
and $\P_p[\partial Q(\frac43 m)\leftrightarrow \partial Q(3m)\text{ in }Z] \geq \P_{p_c}[\partial Q(\frac43 m)\leftrightarrow \partial Q(3m)]\geq c>0$, as proved in \cite{BS15,NTW15}.

\medskip

We proceed to prove \eqref{eq:main:proof}. 
Let $E$ be the event that there exists an open circuit (nearest neighbor path with the same start and end points) 
around $Q(2m)$ contained in $\an(2m,3m)$. It is shown in \cite{NTW15} that $\P_p[E]\geq \P_{p_c}[E]>c>0$ for some $c>0$ independent of $m$. 
Thus, by the FKG inequality, 
\begin{multline*}
\P_p[X\leftrightarrow \partial Q(3m)\text{ in }Z,~Y\leftrightarrow \partial Q(2m)\text{ in }Z,~ E]\\ 
\geq c\cdot \P_p[X\leftrightarrow \partial Q(3m)\text{ in }Z]
\cdot \P_p[Y\leftrightarrow \partial Q(2m)\text{ in }Z].
\end{multline*}
Consider an arbitrary deterministic ordering of all circuits in $\slab$, and for a configuration in $E$, let $\Gamma$ be the minimal (with respect to this ordering) 
open circuit around $Q(2m)$ contained in $\an(2m,3m)$. 
For $W\subset\slab$, let 
\[
\overline W = \{z=(z_1,\ldots,z_d)\in\slab~:~(z_1,z_2,x_3,\ldots,x_d)\in W\text{ for some $x_3,\ldots,x_d$ }\}.
\]
Note that 
\[
\P_p[X\leftrightarrow \partial Q(3m)\text{ in }Z,~Y\leftrightarrow \partial Q(2m)\text{ in }Z,~ E]\leq 
\P_p[X\leftrightarrow \overline\Gamma\text{ in }Z,~Y\leftrightarrow \overline\Gamma\text{ in }Z,~ E].
\]
Thus, to prove \eqref{eq:main:proof}, it suffices to show that for some $C<\infty$, 
\[
\P_p[X\leftrightarrow \overline\Gamma\text{ in }Z,~Y\leftrightarrow \overline\Gamma\text{ in }Z,~ E]\leq C\cdot \P_p[X\leftrightarrow Y\text{ in }Z].
\]
This will be achieved using local modification arguments similar to those in \cite{NTW15}. 
In fact, for the above inequality to hold, it suffices to show that for some $C<\infty$, 
\begin{equation}\label{eq:main:glue}
\P_p[X\leftrightarrow \overline\Gamma\text{ in }Z,~Y\leftrightarrow \overline\Gamma\text{ in }Z,~ E, X\nleftrightarrow Y\text{ in }Z]\leq C\cdot \P_p[X\leftrightarrow Y\text{ in }Z].
\end{equation}
We write the event in the left hand side of \eqref{eq:main:glue} as the union of three subevents satisfying additionally 
\begin{itemize}\itemsep0pt
\item[(a)]
$X\nleftrightarrow \Gamma\text{ in }Z$, $Y\nleftrightarrow \Gamma\text{ in }Z$, 
\item[(b)]
$X\nleftrightarrow \Gamma\text{ in }Z$, $Y\leftrightarrow \Gamma\text{ in }Z$, 
\item[(c)]
$X\leftrightarrow \Gamma\text{ in }Z$, $Y\nleftrightarrow \Gamma\text{ in }Z$. 
\end{itemize}
It suffices to prove that the probability of each of the three subevents can be bounded from above by $C\cdot \P_p[X\leftrightarrow Y\text{ in }Z]$. 
The cases (b) and (c) can be handled similarly, thus we only consider (a) and (b). 

\paragraph{Case (a):} We prove that for some $C<\infty$, 
\begin{equation}\label{eq:main:glue:a}
\P_p\left[
\begin{array}{c}
X\leftrightarrow \overline\Gamma\text{ in }Z,~Y\leftrightarrow \overline\Gamma\text{ in }Z,~ E, X\nleftrightarrow Y\text{ in }Z\\
X\nleftrightarrow \Gamma\text{ in }Z, Y\nleftrightarrow \Gamma\text{ in }Z 
\end{array}
\right]\leq C\cdot \P_p[X\leftrightarrow Y\text{ in }Z].
\end{equation}
Denote by $E_a$ the event on the left hand side. 
It suffices to construct a map $f:E_a\to\{X\leftrightarrow Y\text{ in }Z\}$ such that for some constant $D<\infty$, 
(1) for each $\omega\in E_a$, $\omega$ and $f(\omega)$ differ in at most $D$ edges, 
(2) at most $D$ $\omega$'s can be mapped to the same configuration, i.e., for each $\omega\in E_a$, $|\{\omega'\in E_a:f(\omega') = f(\omega)\}|\leq D$. 
If so, the desired inequality is satisfied with $C = \frac{D}{\min(p_c,1-p_c-\delta))^D}$. 

Take a configuration $\omega\in E_a$. 
Let $U$ be the set of all points $u\in\overline\Gamma$ such that $u$ is connected to $X$ in $Z$ by an open self-avoiding path  
that from the first step on does not visit $\overline{\{u\}}$. For each $u\in U$, choose one such open self-avoiding path and denote it by $\pi_u$. 
Similarly, let $V$ be the set of all points $v\in\overline\Gamma$ such that $v$ is connected to $Y$ in $Z$ by an open self-avoiding path  
that from the first step on does not visit $\overline{\{v\}}$. For each $v\in V$, choose one such open self-avoiding path and denote it by $\pi_v$.

\smallskip

Assume first that we can choose $u\in U$ and $v\in V$ such that $\overline{\{u\}}=\overline{\{v\}}$. 
For such $\omega$'s, the configuration $f(\omega)$ is defined as follows. 
We 
\begin{itemize}\itemsep0pt
\item[(a)] 
close all the edges with an end-vertex in $\overline{\{u\}}$ except for the (unique) edge of $\pi_u$, the (unique) edge of $\pi_v$, and the edges belonging to $\Gamma$, 
\item[(b)] 
open all the edges in $\overline{\{u\}}$ that belong to a shortest path $\rho$ (line segment if $d=3$) between $u$ and $\Gamma$ in $\overline{\{u\}}$, 
\item[(c)] 
open all the edges in $\overline{\{u\}}$ that belong to a shortest path between $v$ and $\Gamma\cup\rho$ in $\overline{\{u\}}$. 
\end{itemize}
Notice that $\omega$ and $f(\omega)$ differ in at most $2d\,(k+1)^{d-2}$ edges. Moreover, 
since $u$, $v$, and $\Gamma$ are all in different open clusters in $\omega$, 
after connecting them by simple open paths as in (b) and (c), no new open circuits are created. Thus, 
the set $\overline{\{u\}}$ can be uniquely reconstructed in $f(\omega)$ as the unique set of the form $\overline{\{z\}}$ 
where $X$ (and $Y$) is connected to $\Gamma$. 

\smallskip

Assume next that $\overline U\cap \overline V = \emptyset$. Choose $u\in U$ and $v\in V$. 
Note that $\overline{\{u\}}$ is not connected to $Y$ in $Z$ and $\overline{\{v\}}$ is not connected to $X$ in $Z$. The configuration $f(\omega)$ is defined as follows. We 
\begin{itemize}\itemsep0pt
\item[(a)] 
close all the edges with an end-vertex in $\overline{\{u\}}\cup\overline{\{v\}}$ except for the edges of $\pi_u$, $\pi_v$, and $\Gamma$, 
\item[(b)] 
open all the edges in $\overline{\{u\}}$ that belong to a shortest path between $u$ and $\Gamma$ in $\overline{\{u\}}$, 
\item[(c)] 
open all the edges in $\overline{\{v\}}$ that belong to a shortest path between $v$ and $\Gamma$ in $\overline{\{v\}}$. 
\end{itemize}
Notice that $\omega$ and $f(\omega)$ differ in at most $4d\,(k+1)^{d-2}$ edges. 
Step (a) of the construction does not alter the paths $\pi_u$ and $\pi_v$. 
Finally, since $u$, $v$, and $\Gamma$ are all in different open clusters in $\omega$, 
after connecting $u$, $v$, and $\Gamma$ by simple open paths as in (b) and (c), no new open circuits are created. Thus, 
the set $\overline{\{u\}}\cup\overline{\{v\}}$ can be uniquely reconstructed in $f(\omega)$ as 
the unique such set where $X$ and $Y$ are connected to $\Gamma$.

The constructed function $f$ satisfies the requirement (1) with $D = 4d\,(k+1)k^{d-2}$ and the requirement (2) with $D = 2^{4d\,(k+1)^{d-2}}$. 
The proof of \eqref{eq:main:glue:a} is complete.

\paragraph{Case (b):} We prove that for some $C<\infty$, 
\begin{equation}\label{eq:main:glue:b}
\P_p\left[
\begin{array}{c}
X\leftrightarrow \overline\Gamma\text{ in }Z,~Y\leftrightarrow \overline\Gamma\text{ in }Z,~ E, X\nleftrightarrow Y\text{ in }Z\\
X\nleftrightarrow \Gamma\text{ in }Z, Y\leftrightarrow \Gamma\text{ in }Z 
\end{array}
\right]\leq C\cdot \P_p[X\leftrightarrow Y\text{ in }Z].
\end{equation}
Denote by $E_b$ the event on the left hand side. As in Case (a), \eqref{eq:main:glue:b} will follow if we 
construct a map $f:E_b\to\{X\leftrightarrow Y\text{ in }Z\}$ such that for some constant $D<\infty$, 
(1) for each $\omega\in E_b$, $\omega$ and $f(\omega)$ differ in at most $D$ edges, 
(2) at most $D$ $\omega$'s are mapped to the same configuration.

Take a configuration $\omega\in E_b$. 
Let $U$ be the set of all points $u\in\overline\Gamma$ such that $u$ is connected to $X$ in $Z$ by an open self-avoiding path 
that from the first step on does not visit $\overline{\{u\}}$. 
For each $u\in U$, choose one such open self-avoiding path and denote it by $\pi_u$.

\smallskip

We first assume that there exists $u\in U$ such that $Y$ is connected to $\Gamma$ in $Z\setminus\overline{\{u\}}$. 
For such $\omega$'s, we define $f(\omega)$ as follows. We 
\begin{itemize}\itemsep0pt
\item[(a)] 
close all the edges with an end-vertex in $\overline{\{u\}}$ except for the edges of $\pi_u$ and $\Gamma$, 
\item[(b)] 
open all the edges in $\overline{\{u\}}$ that belong to a shortest path between $u$ and $\Gamma$ in $\overline{\{u\}}$. 
\end{itemize}
Notice that $\omega$ and $f(\omega)$ differ in at most $2d\,(k+1)^{d-2}$ edges. 
$Y$ is connected to $\Gamma$ in $Z\setminus\overline{\{u\}}$ in the configuration $f(\omega)$. 
Finally, since $u$ and $\Gamma$ are in different open clusters in $\omega$, 
after connecting $u$ and $\Gamma$ by a simple open path as in (b), no new open circuits are created.
Thus, the set $\overline{\{u\}}$ can be uniquely reconstructed in $f(\omega)$ 
as the unique such set where $X$ is connected to $\Gamma$.

\smallskip

Assume next that for any $u\in U$, $Y$ is not connected to $\Gamma$ in $Z\setminus\overline{\{u\}}$. 
Take $u\in U$. There exists $v\in \overline{\{u\}}$ such that $v$ is connected to $Y$ in $Z$ by an open self-avoiding path  
that from the first step on does not visit $\overline{\{v\}}$. 
Choose one such open self-avoiding path and denote it by $\pi_v$. 
For such $\omega$'s, we define $f(\omega)$ exactly as in the first part of Case (a). 
We
\begin{itemize}\itemsep0pt
\item[(a)] 
close all the edges with an end-vertex in $\overline{\{u\}}$ except for the edges of $\pi_u$, $\pi_v$, and $\Gamma$, 
\item[(b)] 
open all the edges in $\overline{\{u\}}$ that belong to a shortest path $\rho$ between $u$ and $\Gamma$ in $\overline{\{u\}}$, 
\item[(c)] 
open all the edges in $\overline{\{u\}}$ that belong to a shortest path between $v$ and $\Gamma\cup\rho$ in $\overline{\{u\}}$. 
\end{itemize}
Notice that unlike in Case (a), it is allowed here that $v\in \Gamma$, but this makes no difference for the construction. 
Indeed, after closing edges as in (a), $Y$ remains connected to $\Gamma$ only if $v\in\Gamma$. 
Thus, after modifying $\omega$ according to (a), either $u$, $v$, and $\Gamma$ are all in different open clusters 
or $v\in\Gamma$ and the clusters of $u$ and $\Gamma$ are different. 
In both cases, after connecting $u$, $v$, and $\Gamma$ by simple open paths as in (b) and (c), no new open circuits are created.  
Thus, the set $\overline{\{u\}}$ can be uniquely reconstructed in $f(\omega)$ as the unique set of the form $\overline{\{z\}}$ 
where $X$ (and $Y$) is connected to $\Gamma$. 

The function $f$ satisfies requirements (1) and (2), and the proof of \eqref{eq:main:glue:b} is complete. 

\smallskip

Since the proof of Case (c) is essentially the same as the proof of Case (b), we omit it. 
Cases (a)-(c) imply \eqref{eq:main:glue}. The proof of Lemma~\ref{l:main} is complete. 
\end{proof}

\begin{remark}
\begin{enumerate}\itemsep0pt
\item[(1)]
Theorem~\ref{thm:slabs} and Remark~\ref{rem:JaraiIIC} can be used to extend various results of J\'arai \cite{Jarai} to slabs. 
For instance, to prove that the local limit of the occupancy configurations around vertices in the bulk of a crossing cluster of large box 
are given by the IIC measures from Theorem~\ref{thm:slabs}. This will be detailed in \cite{Basu}.
\item[(2)]
Using Lemma~\ref{l:main}, one can show that the expected number of vertices of the IIC 
in $Q(n)$ is comparable to $n^2\P[0\leftrightarrow \partial Q(n)]$. 
\item[(3)]
In \cite{DS11}, the so-called multiple-armed IIC measures were introduced for planar lattices, which are supported on configurations with several disjoint infinite open clusters meeting in a neighborhood of the origin. 
These measures describe the local occupancy configurations around outlets of the invasion percolation \cite{DS11} and pivotals for open crossings of large boxes \cite{Basu}. 
It would be interesting to construct multiple-armed IIC measures on slabs, but at the moment it seems quite difficult. 
\end{enumerate}
\end{remark}


\begin{thebibliography}{99}

\bibitem{A97}
M. Aizenman (1997) On the number of incipient spanning clusters.
{\it Nucl. Phys. B} {\bf 485}, 551--582. 
 
\bibitem{Basu}
D. Basu (2016) PhD thesis. In preparation.

\bibitem{BS15}
D. Basu and A. Sapozhnikov (2015)
Crossing probabilities for critical Bernoulli percolation on slabs.
{\it To appear in Ann. Inst. Henri Poincar\'e Probab. Stat.} arXiv:1512.05178.

\bibitem{BS96}
I. Benjamini and O. Schramm (1996)
Percolation beyond $\Z^d$, many questions and a few answers. 
{\it Electron. Comm. Probab.} {\bf 1}, 71--82.

\bibitem{BCKS}
C. Borgs, J. T. Chayes, H. Kesten, and J. Spencer (1999)
Uniform boundedness of critical crossing probabilities implies hyperscaling. 
{\it Random Structures Algorithms} {\bf 15}, 368--413. 

\bibitem{Con85}
A. Coniglio (1985)
Shapes, surfaces and interfaces in percolation clusters.
{\it Proc Les Houches Conf on Physics of Finely Divided Matter}, M. Daoud and N. Boccara (Editors), 
Springer-Verlag, Berlin, 84--109.

\bibitem{CC87}
J. T. Chayes and L. Chayes (1987)
On the upper critical dimension of Bernoulli percolation.
{\it Comm. Math. Phys.} {\bf 113(1)}, 27--48. 


\bibitem{DS11}
M. Damron and A. Sapozhnikov (2011) 
Outlets of $2D$ invasion percolation and multiple-armed incipient infinite clusters. 
{\it Probab. Th. Rel. Fields} {\bf 150}, 257--294.

\bibitem{DCST14}
H. Duminil-Copin, V. Sidoravicius, and V. Tassion (2016) 
Absence of infinite cluster for critical Bernoulli percolation on slabs. 
{\it Comm. Pure Appl. Math.} {\bf 69(7)}, 1397--1411.

\bibitem{HP99}
O. H\"aggstr\"om and Y. Peres (1999)
Monotonicity of uniqueness for percolation on Cayley graphs: all infinite clusters are born simultaneously. 
{\it Probab. Th. Rel. Fields} {\bf 113}, 273--285. 

\bibitem{Hara}
T. Hara (2008) Decay of correlations in nearest-neighbor self-avoiding walk, percolation, lattice trees and animals.
{\it Ann. Probab.} {\bf 36(2)}, 530--593.

\bibitem{HS90}
T. Hara and G. Slade (1990)
Mean-field critical behaviour for percolation in high dimensions.
{\it Comm. Math. Phys.} {\bf 128(2)}, 333--391. 

\bibitem{HvdHH14}
M. Heydenreich, R. van der Hofstad, and T. Hulshof (2014)
High-dimensional incipient infinite clusters revisited. 
{\it J. Stat. Phys.}, {\bf 155}, 966--1025.

\bibitem{vdHJ04}
R. van der Hofstad and A. Jarai (2004) 
The incipient infinite cluster for high-dimensional unoriented percolation. 
{\it J. Stat. Phys.}, {\bf 114(3)}, 625--663.

\bibitem{Jarai}
A. Jarai (2003) 
Incipient infinite percolation clusters in $2D$. 
{\it Ann. Probab.}, {\bf 31(1)}, 444--485.

\bibitem{KestenIIC}
H. Kesten (1986)
The incipient infinite cluster in two-dimensional percolation.
{\it Probab. Th. Rel. Fields}, {\bf 73}, 369--394.

\bibitem{KN11}
G. Kozma and A. Nachmias (2011) Arm exponents in high dimensional percolation.
{\it J. Amer. Math. Soc.}, {\bf 24(2)}, 375--409. 

\bibitem{NTW15}
Ch. Newman, V. Tassion, and W. Wu (2015) 
Critical percolation and the minimal spanning tree in slabs. 
arXiv:1512.09107.

\bibitem{R78}
L. Russo (1978) A note on percolation. 
{\it Z. Wahrsch. Verw. Gebiete}
{\bf 43(1)}, 39--48.

\bibitem{Sch99}
R. Schonmann (1999)
Stability of infinite clusters in supercritical percolation. 
{\it Probab. Th. Rel. Fields} {\bf 113(2)}, 287--300. 

\bibitem{SW78}
P. D. Seymour and D. J. A. Welsh (1978) 
Percolation probabilities on the square lattice. 
{\it Annals of Discrete Mathematics} {\bf 3}, 227--245.

\end{thebibliography}
\end{document}